\documentclass[12pt]{amsart}

\usepackage{amsfonts,amsmath,amsthm,amssymb,latexsym,amsthm,newlfont,enumerate,color}

\newif\ifslide

\theoremstyle{plain}

\ifdefined\spacer
\newtheorem{theorem}{Theorem}
 \else
\newtheorem{theorem}{Theorem}[section]
\fi

\newtheorem{lemma}[theorem]{Lemma}

\newtheorem*{theorem*}{Theorem}
\newtheorem*{corollary*}{Corollary}

\newtheorem{proposition}[theorem]{Proposition}
\newtheorem{definition-lemma}[theorem]{Definition-Lemma}
\newtheorem{question}[theorem]{Question}
\newtheorem{red-question}[theorem]{\textcolor{red}{Question}}

\theoremstyle{definition}
\newtheorem{definition}[theorem]{Definition}

\newtheorem{example}[theorem]{Example}

\def\ideal#1.{I_{#1}}
\def\ring#1.{\mathcal {O}_{#1}}

\def\fring#1.{\hat{\mathcal {O}}_{#1}}
\def\proj#1.{\mathbb {P}(#1)}
\def\pr #1.{\mathbb {P}^{#1}}
\def\dpr #1.{\hat{\mathbb {P}}^{#1}}
\def\af #1.{\mathbb A^{#1}}
\def\Hz #1.{\mathbb F_{#1}}
\def\Hbz #1.{\overline{\mathbb F}_{#1}}
\def\fb#1.{\underset #1 {\times}}
\def\rest#1.{\underset {\ \ring #1.} \to \otimes}
\def\au#1.{\operatorname {Aut}\,(#1)}
\def\deg#1.{\operatorname {deg } (#1)}

\def\pic#1.{\operatorname {Pic}\,(#1)}
\def\pico#1.{\operatorname{Pic}^0(#1)}
\def\picg#1.{\operatorname {Pic}^G(#1)}
\def\ner#1.{NS (#1)}
\def\rdown#1.{\llcorner#1\lrcorner}
\def\rfdown#1.{\lfloor{#1}\rfloor}
\def\rup#1.{\ulcorner{#1}\urcorner}
\def\rcup#1.{\lceil{#1}\rceil}

\def\n1#1.{\operatorname {N_1}(#1)}  
\def\cn1#1.{\overline{\operatorname {N^1}(#1)}} 
\def\cone#1.{\operatorname {NE}(#1)}     
\def\ccone#1.{\overline{\operatorname {NE}}(#1)}
\def\none#1.{\operatorname {NF}(#1)}
\def\cnone#1.{\overline{\operatorname {NF}}(#1)}
\def\mone#1.{\operatorname {NM}(#1)} 
\def\cmone#1.{\overline{\operatorname {NM}}(#1)}

\def\coef#1.{\frac{(#1-1)}{#1}}
\def\vit#1.{D_{\langle #1 \rangle}}
\def\mm#1.{\overline {M}_{0,#1}}
\def\H1#1.{H^1(#1,{\ring #1.})}
\def\ac#1.{\overline {\mathbb F}_{#1}}

\def\adj#1.{\frac {#1-1}{#1}}
\def\spn#1.{\overline{#1}}
\def\pek#1.#2.{\Cal P^{#1}(#2)}
\def\plk#1.#2.{\Cal P^{\leq #1}(#2)}
\def\ev#1.{\operatorname{ev_{#1}}}
\def\ilist#1.{{#1}_1,{#1}_2,\dots}
\def\bminv#1.{(\nu_1,s_1;\nu_2,s_2;\dots ;\nu_{#1},s_{#1};\nu_{r+1})}
\def\zinv#1.{(\nu_1,s_1;\nu_2,s_2;\dots ;\nu_{#1},s_{#1};0)}
\def\iinv#1.{(\nu_1,s_1;\nu_2,s_2;\dots ;\nu_{#1},s_{#1};\infty)}

\def\scr #1.{\mathcal #1}

\def\llist#1.#2.{{#1}_1,{#1}_2,\dots,{#1}_{#2}}
\def\ulist#1.#2.{{#1}^1,{#1}^2,\dots,{#1}^{#2}}
\def\lomitlist#1.#2.{{#1}_1,{#1}_2,\dots,\hat {{#1}_i}, \dots, {#1}_{#2}}
\def\lomitlistz#1.#2.{{#1}_0,{#1}_1,\dots,\hat {{#1}_i}, \dots, {#1}_{#2}}
\def\loc#1.#2.{\Cal O_{#1,#2}}
\def\fderiv#1.#2.{\frac {\partial #1}{\partial #2}}
\def\deriv#1.#2.{\frac {d #1}{d #2}}

\def\map#1.#2.{#1 \longrightarrow #2}
\def\rmap#1.#2.{#1 \dasharrow #2}
\def\emb#1.#2.{#1 \hookrightarrow #2}
\def\non#1.#2.{\text {Spec }#1[\epsilon]/(\epsilon)^{#2}}
\def\Hi#1.#2.{\text {Hilb}^{#1}(#2)}
\def\sym#1.#2.{\operatorname {Sym}^{#1}(#2)}
\def\Hb#1.#2.{\text {Hilb}_{#1}(#2)}
\def\Hm#1.#2.{\Hom_{#1}(#2)}
\def\prd#1.#2.{{#1}_1\cdot {#1}_2\cdots {#1}_{#2}}
\def\Bl #1.#2.{\operatorname {Bl}_{#1}#2}
\def\pl #1.#2.{#1^{\otimes #2}}
\def\mgn#1.#2.{\overline {M}_{#1,#2}}
\def\ialist#1.#2.{{#1}_1 #2 {#1}_2, #2\dots}
\def\pair#1.#2.{\langle #1, #2\rangle}
\def\vandermonde#1.#2.{\left|
\begin{matrix}
1 & 1 & 1 & \dots & 1\\
{#1}_1 & {#1}_2 & {#1}_3 & \dots & {#1}_{#2}\\
{#1}_1^2 & {#1}_2^2 & {#1}_3^2 & \dots & {#1}_{#2}^2\\
\vdots & \vdots & \vdots & \ddots & \vdots\\
{#1}_1^{#2-1} & {#1}_2^{#2-1} & {#1}_2^{#2-1} & \dots & {#1}_{#2}^{#2-1}\\
\end{matrix}
\right|
}
\def\vandermondet#1.#2.{\left|
\begin{matrix}
1 & {#1}_1   & {#1}_1^2 & \dots & {#1}_1^{#2-1}\\
1 & {#1}_2   & {#1}_2^2 & \dots & {#1}_2^{#2-1}\\
1 & {#1}_3   & {#1}_3^2 & \dots & {#1}_3^{#2-1}\\
\vdots & \vdots & \vdots & \ddots & \vdots\\
1 & {#1}_{#2}& {#1}_{#2}^2 & \dots & {#1}_{#2}^{#2-1}\\
\end{matrix}
\right|
}
\def\gr#1.#2.{\mathbb{G}(#1,#2)}

\def\alist#1.#2.#3.{{#1}_1 #2 {#1}_2 #2\dots #2 {#1}_{#3}}
\def\zlist#1.#2.#3.{#1_0 #2 #1_1 #2\dots #2 #1_{#3}}
\def\lomitlist30#1.#2.#3.{{#1}_0,{#1}_1 #2 \dots #2\hat {{#1}_i} #2\dots #2 {#1}_{#3}}
\def\lmap#1.#2.#3.{#1 \overset{#2}{\longrightarrow} #3}
\def\mes#1.#2.#3.{#1 \longrightarrow #2 \longrightarrow #3}
\def\ses#1.#2.#3.{0\longrightarrow #1 \longrightarrow #2 \longrightarrow #3 \longrightarrow 0}
\def\les#1.#2.#3.{0\longrightarrow #1 \longrightarrow #2 \longrightarrow #3}
\def\res#1.#2.#3.{#1 \longrightarrow #2 \longrightarrow #3\longrightarrow 0}
\def\Hi#1.#2.#3.{\text {Hilb}^{#1}_{#2}(#3)}
\def\ten#1.#2.#3.{#1\underset {#2}{\otimes} #3}
\def\lomitlist30#1.#2.#3.{{#1}_0 #2 {#1}_1 #2 \dots #2 \hat {{#1}_i} #2 \dots #2 {#1}_{#3}}
\def\mderiv#1.#2.#3.{\frac {d^{#3} #1}{d #2^{#3}}}

\def\Hom{\operatorname{Hom}}

\def\dim{\operatorname{dim}}

\def\deg{\operatorname{deg}}

\def\det{\operatorname{det}}

\def\rk{\operatorname{rk}}

\def\rest{\operatorname{res}}

\def\vol{\operatorname{vol}}

\def\C{\mathbb C}

\def\e{\Cal E}

\def\e1{E_1}
\def\e2{E_2}

\def\Q{\mathbb Q}
\def\Z{\mathbb Z}

\def\mapdown#1{\big\downarrow\rlap{$\vcenter{\hbox{$\scriptstyle#1$}}$}}

\def\mapse#1{
{\vcenter{\hbox{$\mathop{\smash{\raise1pt\hbox{$\diagdown$}\!\lower7pt
\hbox{$\searrow$}}\vphantom{p}}\limits_{#1}\vphantom{\mapdown{}}$}}}}

\def\VR#1.{height#1pt&\omit&&\omit&&\omit&&\omit&&\omit&\cr}

\def\VRT#1.{height#1pt&\omit&&\omit&\cr}

\usepackage{hyperref}
\usepackage{amsmath}
\usepackage{graphicx}

\title{Forms and complex manifolds}

\author{Cinzia Bisi}
\address{Department of Mathematics and Computer Science\\ 
 University of Ferrara\\ Via Machiavelli, 30 \\ Ferrara 44121, Italy}
 \email{bsicnz@unife.it}

\author{Paolo Cascini}
\address{Department of Mathematics\\
Imperial College London\\
180 Queen's Gate\\
London SW7 2AZ, UK}
\email{p.cascini@imperial.ac.uk}

\author{Luca Tasin}
\address{Department of Mathematics " Federigo Enriques " \\ University of Milano \\  Via Saldini,  50 \\ Milano 20133, Italy} 
\email{luca.tasin@unimi.it}

\begin{document}

\begin{abstract}
We study the intersection form 
$F_X$ on the second cohomology group $H^2(X, \mathbb{Z})$ of a compact Kähler manifold $X$ of  dimension $n$. 
Although the structure of $F_X$ is relatively well understood in dimensions two and three, 
much less is known for $n \geq 4$. 
We investigate the fundamental properties of $F_X$  in higher dimensions 
and discuss several applications to birational geometry. 
Finally, we present a number of open problems concerning the relationship 
between birational invariants and topological invariants of Kähler manifolds.
\end{abstract}
\maketitle
\setcounter{tocdepth}{1}

\tableofcontents

\section{Introduction}

Let $X$ be a compact K\"ahler manifold of dimension $n$. One of the fundamental objects associated with $X$ is the natural  form $F_X$ on $H^2(X,\mathbb{Z})$ of degree $n$, induced by the cup product structure. This form serves as a topological invariant of $X$ and plays a crucial role in understanding the intersection theory of higher-dimensional algebraic varieties.

For $n=2$, the quadratic form $F_X$ is central to the theory of algebraic surfaces, yielding results in classification theory and the study of their moduli spaces. When $n=3$, the  study of the cubic form $F_X$ was initiated in \cite{OV95}, leading to numerous applications in algebraic geometry, differential geometry, and dynamical systems. Recent advancements have further deepened our understanding of $F_X$ and its role in various geometric problems \cite{KW, BCT16, CT18, ST19, ST20, Wilson21, Wilson21b}.

Despite these advances, relatively little is known about the structure and applications of $F_X$ in dimensions $n \geq 4$. The goal of this paper is to initiate a study of $F_X$ in higher dimensions, exploring its fundamental properties and potential applications to birational geometry and topology. In particular, we examine how the rank  of the hypermatrix $\mathcal H_{F_X}$ of the $(n-1)$-th derivatives of $F_X$ influence the behavior of divisorial contractions (see Section \ref{s:intersection} for the definition of $\mathcal H_{F_X}$ and some of its basic properties).

\medskip 

Our main result establishes the existence of a finite set of cohomology classes that control the exceptional divisors of divisorial contractions:

\begin{theorem}\label{thm_main}
Let $M$ be a closed topological manifold of dimension $2n$ and let $b = \dim H^2(M,\mathbb{C})$.

Then there exist non-zero elements $e_1,\ldots,e_q \in H^2(M,\mathbb{C})$ with $q \leq b+1$ such that if $X$ is a smooth complex projective variety of dimension $n$ with underlying topological space $M$ and $f\colon X \to Y$ is a divisorial contraction to a point with exceptional divisor $E$, then there exists $i \in \{1, \dots, q\}$ such that $[E] = e_i$.
\end{theorem}

This result provides a topological bound on the possible exceptional divisors arising from some of the steps of the Minimal Model Program, establishing a new link between topology and birational geometry. It is a consequence of Theorem \ref{thm:W_F}, where we prove the finiteness of rank one points of $\mathcal H_F$ where $F$ is a non-degenerate $n$-form, a result of independent interest. 

\medskip 

The structure of the paper is as follows. In Section \ref{s_preliminaries}, we introduce the necessary background on tensors and forms. Sections \ref{s_low} and \ref{s:intersection} are devoted to the study of the tensor arising from the intersection form on a closed topological manifold $M$ of dimension $2n$ with $n\ge 3$. We show, in particular, that, up to a non-zero scalar, there are at most finitely many  classes $e\in H^2(M,\mathbb C)$ such that 
$e^n\neq 0$ and the rank of the Hessian of $F$ at $e$ is one.
(cf. Theorem \ref{thm:W_F}). Section \ref{s_blowups} establishes our main theorem by examining the rank  of the associated  form in the presence of a divisorial contraction. In Section \ref{s_volume},  we present some related open problems. 

\medskip

\textbf{Acknowledgements:}  The first and third authors are members of the GNSAGA - Istituto Nazionale di Alta Matematica.
The first author was also partially supported by the PRIN ``Vartietà reali e complesse: geometria, topologia e analisi armonica". 
The second author is partially supported by a Simons collaboration grant  and would like to thank the National Center for Theoretical Sciences in Taiwan and Professor Jungkai Chen for
their  hospitality, where some of the work for this paper was completed.
The third author is partially supported by the PRIN2020 research grant ``2020KKWT53”.
The authors would also like to thank the referee for careful reading the paper and for several useful comments.

\section{Preliminaries}\label{s_preliminaries}
We work over the field of complex numbers $\mathbb C$. A \emph{form} over $\mathbb C$ is a homogeneous polynomial $F\in \mathbb C[x_0,\ldots,x_b]$.

\subsection{Tensors} Let $a_1,\dots,a_n$ be positive integers. A \emph{tensor} of type $a_1\times \ldots \times a_n$ is a multilinear map 
$$T\colon \mathbb C^{a_1}\times \ldots \times \mathbb C^{a_n}\to \mathbb C.$$
For any positive integer $a$, denote $[a]:=\{ 1,2,\ldots, a \}$. Then a tensor $T$ of type  $a_1\times \ldots \times a_n$ is determined uniquely by a function 
$$\tilde T\colon [a_1]\times \cdots \times [a_n] \to \mathbb{C}.$$
We will refer to $\tilde T$ as the \emph{hypermatrix associated to $T$}. 
Given positive integers 
$a_j^{'} \le a_j$ for $1 \le j \le n,$ and strictly increasing functions $f_j\colon [a_j^{'}] \to [a_j]$,
we define a \emph{sub-tensor} $T'$ of type $a'_1\times\ldots\times a'_n$ so that if $\tilde T'$ is the hypermatrix associated to $T'$ then
$$\tilde T'=\tilde T(f_1,\dots,f_n)\colon [a'_1]\times \cdots \times [a'_n] \to \mathbb{C}.$$
Note that if $v_i\in \mathbb C^{a_i}$ for $i=1,\dots,n$, then there exists a natural tensor of type $a_1\times \ldots \times a_n$ defined by 
$v_1\otimes \ldots \otimes v_n$.

\begin{definition}
A nonzero tensor $T$ of type  $a_1\times \ldots \times a_n$ has {\it rank one} if there are nonzero vectors $v_i \in \mathbb{C}^{a_i}$ such that $T=v_1 \otimes \cdots \otimes v_n $. We define the
{\it rank} of a nonzero tensor $T$ ($=\rk(T)$) to be the minimum positive integer $r$
such that there exist $r$ tensors $T_1, \cdots , T_r$ of rank one with 

$$T = T_1 + \cdots + T_r .$$
\end{definition}

Note that the rank of any sub-tensor of $T$ cannot be larger than the rank of $T$. 
\begin{definition}\label{d_cubic}
A \emph{cubic tensor} $T$ is a tensor of type $d^{\times n} =d \times \cdots \times d$. Let $A$ be the hypermatrix associated to $T$. 
We denote by $\det A$ the \emph{hyperdeterminant} of  $A$ (see \cite[Chapter 14]{GKZ} for the definition and some of its main properties). In particular, $\det A$  is a polynomial  in the entries of $A$.    

\end{definition}

The following theorem is useful to compute the rank of a tensor:
\begin{theorem}{\cite[Thm. 3.1.1.1]{Landsberg-book}} \label{thm:slices}
Let $A_1,\dots,A_n$ be $l\times m$ matrices and consider the  tensor $A=[A_1,\ldots,A_n]$ of type $l\times m\times n$. 

Then $\rk(A)$ is equal to  the minimum number of rank one matrices needed to span a vector space containing $\langle A_1,\ldots, A_n \rangle$.
\end{theorem}

In particular, we have the following:

\begin{lemma}\label{lem:trick}
Let 
$A_0,\dots,A_q$ be linearly independent  $l\times m$ matrices and consider the  tensor $A=[A_0,\ldots,A_q]$ of type $l\times m\times (q+1)$. 
Assume that for  any $q$-uple $(\mu_1, \ldots, \mu_q) \in \C^q$ the rank of the matrix $A_0 + \sum_{j=1}^q \mu_j A_j$ is at least $t$. 

Then $\rk(A) \ge q+t$.	 
\end{lemma}
\begin{proof}
Let $r=\rk(A)$ and let $B_1,\ldots, B_r$ be rank one matrices such that 
$$\langle A_0,\ldots, A_q\rangle\subset \langle B_1,\ldots, B_r\rangle,$$
as in Theorem \ref{thm:slices}. 
Thus, for any $j=0,\ldots,q$, we may write
$$
A_j= \sum_{i=1}^r \lambda_{ij}B_i
$$
for $\lambda_{ij} \in \C$ for $i=1,\ldots,r$ and $j=0,\dots,q$. Since $A_1,\ldots, A_q$ are linearly independent, after possibly reordering the $B_i$'s, we may assume that the $q\times q$-matrix  
$$L=(\lambda_{ij})_{i,j=1,\ldots,q}$$
 is of maximal rank $q$. Thus, the linear system 
$$L \cdot X =(-\lambda_{10},-\lambda_{20},\ldots,-\lambda_{q0})^t$$ 
admits a solution $(\mu_1,\ldots,\mu_q)\in \mathbb C^q$. 
It follows that 
 $$A_0 + \sum_{i=1}^q \mu_i A_i=\sum_{j=q+1}^{r} s_j B_j$$ for some  $s_{q+1},\ldots,s_r \in \C$.
 Since the rank of $\sum_{j=q+1}^{r} s_j B_j$ is at most $r-q$, our assumption implies that $r-q \ge t$  and the claim follows. 
\end{proof}

\section{Points of low rank}\label{s_low}

\begin{definition}
	Given a form $F \in \C[x_0,\ldots,x_b]$ of degree $n \ge 3$ we consider the tensor 
of type $(b+1)^{\times (n-1)}$
		defined by the hypermatrix $\mathcal H_F$ of the $(n-1)$-th order derivatives, whose entries are linear forms in $x_0, \ldots, x_b$. 
	We say that $F$ is \emph{honest} if for any non-zero element $v \in \mathbb C^{b+1}$, we have  $\mathcal H_F(v) \ne 0$. 
    We say that $F$ is  \emph{non-degenerate} if $ \det \mathcal H_F(v) \ne 0$ for some 
 non-zero $v \in \mathbb C^{b+1}$
 (note that this definition is consistent with \cite[Page 7929]{CT18} but not with \cite[14.1.A]{GKZ}).
\end{definition}

Note that the locus $p \in \mathbb P^b$ such that  $\mathcal H_F(p) = 0$ is well defined. By \cite[Remark 6.3.5]{Bocci-Chiantini}, the rank $\rk \mathcal H_F(p)$ is also well-defined.

\begin{example}\label{ex:degenerate}
	Consider the  form 
	$$
	F(x_0,\ldots,x_4)=\frac{x_0x_1^2}{2} + x_1x_3x_4 + \frac{x_2x_3^2}{2}.
	$$ 
	Then
	
	$$
	\mathcal H_F=\begin{pmatrix}
	0   & x_1 & 0   & 0   & 0   \\
	x_1 & x_0 & 0   & x_4 & x_3 \\
	0   & 0   & 0   & x_3 & 0   \\ 
	0   & x_4 & x_3 & x_2 & x_1 \\
	0   & x_3 & 0   & x_1 & 0 
	\end{pmatrix}
	$$
Thus, the form $F$ is honest, but it is degenerate since $\det \mathcal H_F$ is identically zero. 
\end{example}

Vaguely speaking, an honest form can be characterised by the fact that it  depends on all the variables. More precisely, we have:
\begin{lemma}\label{lem:H_F(p)=0}
	Let $F \in \C[x_0,\ldots,x_b]$ be a form of degree $n$ and let 
    $p:=[1,0,\ldots,0] \in \mathbb P^b$.  

    Then $\mathcal H_F(p) = 0$ if and only if $x_0$ does not appear in the expression of $F$.  In particular, a non-degenerate form is honest.
\end{lemma}	

\begin{proof}
	Assume first that $\mathcal H_F(p) = 0$. By  the Euler formula for homogeneous polynomials we have that
	$$
	(n-k)\partial_{i_1 \ldots i_k}F(x_0,\ldots,x_b)= \sum_{\ell=0}^b x_\ell \partial_{\ell, i_1 \ldots i_k}F(x_0, \ldots, x_b)
	$$
	for any $k=0,\ldots, n-1$ and $i_j \in \{0,\ldots, b\}$. 
    Proceeding by induction and using the fact that $\mathcal H_F(p) = 0$, we get that $\partial_{i_1 \ldots i_k}F(p)=0$ for any $k=0,\ldots, n-1$ and $i_j \in \{0,\ldots, b\}$. Thus, $x_0$ does not appear in $F(x_0, \ldots, x_n)$.
	The converse is a simple computation.   		

\medskip

Assume now that $F(x_0, \ldots, x_n)$ is a non-degenerate form. If by contradiction there exists $p\in \mathbb P^b$ such that $\mathcal H_F(p) = 0$ then, up to a  change of coordinates, we may assume $p=[1,0,\ldots,0]$. This means that $F$ does not depends on $x_0$ and so the face of $\mathcal H_F(p)=0$  corresponding to $\partial_0$ is trivial. By \cite[Corollary XIV.1.5(d)]{GKZ}, it follows that $\det \mathcal H_F$ is identically zero, which is a contradiction.
\end{proof}

\begin{definition}
Let $F \in \C[x_0,\ldots,x_b]$ be a form of degree $n$. We denote by $W_F$ the set of points $p \in \mathbb P^b$ such that $\rk \mathcal H_{F}(p) =1$.
\end{definition}

Note that $W_F$ is a closed subset of $\mathbb P^b$ by \cite[Theorem 6.4.13]{Bocci-Chiantini}.
The aim of this section is to prove the following:

\begin{theorem}\label{thm:W_F}
Let $F \in \C[x_0, x_1, \ldots,x_b]$ be a form of degree $n \ge 3$. 
\begin{enumerate}
\item If $F$ is honest, then $W_F \cap \{F \ne 0\}$ is a finite set, and more precisely, 
$$
|W_F  \cap \{F \ne 0\} | \le b+1.
$$

\item If $F$ is non-degenerate, then $W_F$ is a finite set.
\end{enumerate}
\end{theorem}

If $F$ is honest but degenerate, then $W_F$ may be of positive dimension as the following example shows:

\begin{example} \label{ex:dimW_F>0}
Let  
$$
F(x_0,\ldots,x_4)=\frac{x_0x_1^2}{2} + x_1x_3x_4 + \frac{x_2x_3^2}{2}.
$$	
be the cubic form as in Example \ref{ex:degenerate}. Then $W_F=\{x_1=x_3=x_0x_2-x_4^2=0\}$  is a  plane conic. Note that $W_F \subset \{F=0\}$.
\end{example}

To prove Theorem \ref{thm:W_F} we are going to use the following Propositions \ref{prop:F(p)ne0} and \ref{prop:F(p)=0}, which show that if there exists $p \in W_F$, then we can write $F$ in a special form, depending on whether $F(p) \ne 0$ or $F(p)=0$.

\begin{lemma} \label{lem:rank1}
	Let $T$ be a tensor of type $a_1 \times \cdots \times a_d$. Assume $\rk T=1$. Let $\{k_1,\ldots,k_{d-2}\}$ and $\{h_1,\ldots,h_{d-2}\}$ be two sets of indices such that  $1 \le k_i, h_i \le a_i$ for any $i$. Then the matrices $\tilde T_{k_1,\ldots k_{d-2},\bullet,\bullet}$ and $\tilde T_{h_1,\ldots h_{d-2},\bullet,\bullet}$ are proportional. 
\end{lemma}
\begin{proof}
Consider the $2\times a_{d-1} \times a_d$ subtensor $S$ of $T$ given by $T_{k_1,\ldots k_{d-2},\bullet,\bullet}$ and $T_{h_1,\ldots h_{d-2},\bullet,\bullet}$. We have $\rk S \le \rk T=1$. If $\rk S=0$, then $S$ is trivial and we are done. If $\rk S=1$ then any two of its parallel faces must be proportional (see for instance \cite[Proposition 6.4.6]{Bocci-Chiantini}).
\end{proof}

\begin{proposition}\label{prop:F(p)ne0}
Let $F \in \C[x_0, x_1, \ldots,x_b]$ be a form of degree $n\ge 3$. Let $p \in W_F,$ with $F(p) \ne 0$. Then after a change of coordinates, we may write  $p=[1:0: \ldots 0]$ and $F=x_0^n + G(x_1, \cdots, x_b)$.	

Moreover, if $F$ is honest, then $G$ is honest.
\end{proposition}	

\begin{proof}
After taking a change of coordinates, we may write $p=[1,0, \ldots , 0]$ and  $$
F= \frac{x_0^n}{n!} + \frac{x_0^{n-1}}{(n-1)!} (c_1 x_1 + \ldots +c_b x_b) + \frac{x_0^{n-2}}{2(n-2)!}(x_1^2 + \ldots + x_r^2) + 
R(x_0,x_1, \ldots, x_b), 
$$ 
where $r \in \{0, \ldots, b\}$ and $R$ is a form such that the degree with respect to $x_0$ of every monomial in $R$ is at most  $n-3$.

The submatrix of $\mathcal H_{F}(p)$ which corresponds to 
 $
 \left(\partial_j \partial_\ell \partial_0^{n-3}F(p) \right)_{j,\ell = 0, \ldots,  b }
 $ 
is
$$
A=\begin{pmatrix}
1 & c_1 & \ldots & c_r & c_{r+1} & \ldots & c_b \\
c_1 & 1  & 0 & \ldots & \ldots & \dots & 0 \\
\vdots & 0 & 1  & 0 & \cdots & \cdots & \vdots \\
c_r & 0 & \ldots & 1 & 0 & \ldots & 0  \\
\vdots 
&\vdots & \vdots  &\vdots & 0 & \ldots & 0 \\
\vdots&\vdots & \vdots  &\vdots & \vdots & \ddots & \vdots \\
c_b & 0 & \cdots & \cdots & \cdots & \cdots & 0 \\
\end{pmatrix}.
$$

Since $p \in W_F$ and $A\neq 0$, we have that $\rk A = 1$. This implies that $r \in \{ 0, 1 \}$ and $c_i=0$ for $i \ge 2$.

Thus, there exist $\lambda,c\in \mathbb C$ such that after applying  the change of coordinate $x_0 \mapsto x_0 -\lambda x_1$ which fixes the point $p=[1: 0 : \ldots : 0 ]$,  we may write 
$$
F=\frac{x_0^n}{n!} + \frac{c}{2(n-2)!} x_0^{n-2}x_1^2 +  R'(x_0,\dots,x_n),
$$
where $R'$ is a form such that the degree with respect to $x_0$ of every mononmial in $R'$ is at most  $n-3$.
In particular, $A$ becomes
$$
\begin{pmatrix}
1 & 0 & 0 & \ldots  \\
0 & c  & 0 & \ldots  \\
0 & 0 & 0 & \ldots  \\
\vdots & \vdots & \vdots  &\ddots  \\
\end{pmatrix}.
$$
Since $\rk A =1$, it follows that $c=0$ and we may write
$$
F= \frac{x_0^n}{n!} + x_0^{n-3}G_3(x_1,\ldots,x_b) + \ldots +x_0G_{n-1}(x_1,\ldots,x_b) + G(x_1,\ldots,x_b),
$$
where, for any $i=3,\ldots, n-1$, $G_i$ is a homogeneous polynomial of degree $i$ in $x_1,\ldots,x_b$. In particular, if $n=3$ then $F$ is in the desired form.

\medskip 

We distinguish two cases. Assume first that $n=4$. We assume by contradiction that there exists a monomial  $x_0x_{i_1}x_{i_2}x_{i_3}$ of $x_0G_3$ with $i_1,i_2,i_3>0$, 
and we consider the submatrix of $\mathcal H_F(p)$ given by
$$
B:=\left(\partial_j \partial_\ell  \partial_{i_3}F(p) \right)_{j,\ell = 0, \ldots,  b }.
$$ 
By Lemma \ref{lem:rank1}, $A$ and $B$ are proportional, but $B_{00}=0$ and so $B=0$. This gives a contradiction because $B_{i_1i_2} \ne 0$. Hence $G_3=0$ and this proves that if  $n=4$ then $F$ is in the desired form. 

\medskip 

Assume now that $n \ge 5$.
We proceed by induction on $k$ to show that  $G_k=0$ for any $k=3,\ldots,n-1$.
Assume by contradiction that there exists a monomial  $x_0^{n-3}x_{i_1}x_{i_2}x_{i_3}$ of $x_0^{n-3}G_3$, with $i_1,i_2,i_3>0$ and consider the submatrix of $\mathcal H_F(p)$ given by
$$
B:=\left(\partial_j \partial_\ell \partial_0^{n-5}\partial_{i_2} \partial_{i_3}F(p) \right)_{j,\ell = 0, \ldots,  b }.
$$ 

We have that  $B_{00}= \partial_0^{n-3}\partial_{i_2} \partial_{i_3}F(p)=0$

and $B_{0i_1}=B_{i_10}=\partial_0^{n-4}\partial_{i_1}\partial_{i_2} \partial_{i_3}F(p)\ne 0$, which implies that $\rk B\ge 2$, a contradiction. Thus, $G_3=0$. 

We now assume that $k\ge 4$. Assume by contradiction that there exists a monomial  $x_0^{n-k}x_{i_1}\cdot \ldots\cdot x_{i_k}$ in $x_0^{n-k}G_k$, with $i_1,\dots,i_k>0$. 
Consider the submatrix of $\mathcal H_F(p)$ given by
$$
B:=\left(\partial_j \partial_\ell \partial_0^{n-k-2}\partial_{i_2} \cdots \partial_{i_k}F(p) \right)_{j,\ell = 0, \ldots,  b }.
$$ 
Since $G_i=0$ for $i=3,\ldots, k-1$, we have that $B_{00}= \partial_0^{n-k}\partial_{i_2} \cdots \partial_{i_k}F(p)=0$ and $B_{0i_1}=B_{i_10}=\partial_0^{n-k-1}\partial_{i_1}\cdots \partial_{i_k}F(p) \ne 0$, which implies that $\rk B\ge 2$, a contradiction. Thus, $G_k=0$ and hence
$$
F= \frac{x_0^n}{n!} + G(x_1,\ldots,x_b),
$$
as claimed.

\smallskip

Assume now that $F$ is honest and that there exists $q=[q_1, q_2, \ldots ,q_b]$ such that $\mathcal H_G(q)=0$. Then $\mathcal H_F([0,q_1,q_2,\ldots ,q_b])=0$, which gives $q_1=\ldots=q_b=0$ and so $G$ is honest as well. 
\end{proof}

\begin{lemma}\label{lem:x_0H}
Let $H\in \mathbb C[x_1,\dots,x_b]$ be a form of degree $n-1$ for some $n\ge 3$ and let 
$ 
R=x_0H(x_1,\ldots, x_b)
$. Assume that $\rk \mathcal H_R(p)=1$, where $p:=[1:0:\ldots:0]$. 

Then, after a change of coordinates in $(x_1,\ldots,x_b)$, we may write
 $$
 R=x_0x_1^{n-1}.
 $$	
\end{lemma}

\begin{proof}
Consider the tensor $T$ of the $(n-1)$-th order derivatives of $H$.  

Since $H$  is non-zero, it follows that $T$ is non-zero. Moreover $T$  is a subtensor of $\mathcal H_R(p)$ and, therefore, we have that  $\rk T=1$. Thus, since $T$ is symmetric, $T=v^{\otimes (n-1)}$ for some vector $v \in \mathbb C^b$ (see \cite[Proposition 7.2.1]{Bocci-Chiantini}).
After taking a change of coordinates, we may assume $v=(1,0, \ldots,0)$ which, after multiplying by a scalar, implies that $H=x_1^{n-1}$. 
\end{proof}

\begin{proposition}\label{prop:F(p)=0}
Let $F \in \C[x_0, x_1, \ldots,x_b]$ be a  form of degree $n\ge 3$. Let $p \in W_F,$ with $F(p)=0.$ Then, after a change of coordinates, we may assume $p=[1,0, \ldots,0]$ and $F=x_0x_1^{n-1} + G(x_1, \ldots, x_b).$
\end{proposition}

\begin{proof}
We can assume that $p=[1,0,\cdots,0]$ and write
	$$
	F=x_0^{n-1}\cdot R_1 + x_0^{n-2} \cdot R_2 + \cdots + x_0 \cdot R_{n-1} +R_n.
	$$ 
	where  $R_i=R_i(x_1, \cdots, x_b)$ is a form of degree $i$, for $i=1,\dots,n$.  We first show that $R_i=0$ for $i=1, \ldots, n-2$. The proof is by induction on $i$.
	
We may write  $R_1=	c_1 x_1+ \cdots + c_b x_b$. Let  $A$ be the submatrix of $\mathcal H_{F}(p)$ given by 
$$
A=\left(\partial_j \partial_\ell \partial_0^{n-3}F(p)\right)_{j, \ell = 0, \ldots,  b }.
$$ 
Since $\rk A \le 1$ and 	
$
A_{00}= \partial_0^{n-1}F(p) =0
$
we must have 	 
$$
0=A_{0j}=A_{j0}=\partial_j \partial_0^{n-2} F(p)=(n-1)!c_j
$$
for all $j=1, \cdots , b$, which implies $R_1=0$.
	
Now, let $2 \le k \le n-2$ and suppose that $R_{k-1}=0$. We are going to show that $R_k=0$.	
Assume by contradiction that there exists a monomial  
$x_0^{n-k} x_{i_1} \cdot \ldots\cdot x_{i_k}$ in $x_0^{n-k}R_k$ with $i_1,\dots,i_k>0$. Consider the submatrix $A$ of $\mathcal H_{F}(p)$ given by

$$
A=\left(\partial_j \partial_\ell \partial_{i_2} \cdots \partial_{i_k}\partial_0^{n-k-2}F(p)\right)_{j, \ell = 0, \ldots,  b }.
$$ 
Since $R_{k-1}=0$, we have that 
$
A_{00}=\partial_{i_2} \cdots \partial_{i_k}\partial_0^{n-k}F(p)=0.
$
Since $x_0^{n-k} x_{i_1} \cdot \ldots\cdot x_{i_k}$ is a monomial in $x_0^{n-k}R_k$, we have that 
$$
A_{0i_1}=A_{i_10}=\partial_{i_1} \partial_{i_2} \cdots \partial_{i_k}\partial_0^{n-k-1}F(p)\neq 0,
$$
contradicting the fact that $\rk A\le 1$. We hence conclude that $R_k=0$, as claimed.

\medskip 

Thus, we may write
	$$
	F=x_0 \cdot H(x_1, \cdots , x_b) + G(x_1, \cdots , x_b).
	$$
	where $H$ is a form of degree $n-1$ such that if $R:=x_0\cdot H$ then $\rk \mathcal H_R(p)=\rk \mathcal H_F(p)=1$.
    We can then conclude by applying Lemma \ref{lem:x_0H}.
\end{proof}

\begin{proof}[Proof of Theorem \ref{thm:W_F}]
We start by  proving $(1)$.  The proof is by induction on $b$. For $b=0$, the statement is trivial, so we assume $b \ge 1$.
Let $p \in W_F$ such that $F(p) \ne 0$. 
By Proposition \ref{prop:F(p)ne0}, after a change of coordinates, we may assume $p=[1,0,\ldots ,0]$ and
$$
F=x_0^n + G(x_1,\cdots, , x_b),
$$
where $G$ is an honest form of degree $n$.

We claim that for any $q \in W_F$ with $q \ne p$, we have $q \in \{x_0=0\}$. Suppose by contradiction that $q=[1,q_1, \ldots ,q_b ]$. 
If $\mathcal H_G([q_1, \ldots,q_b ]) \ne 0 $, then $\rk \mathcal H_F(q) \ge 2$. Hence  $\mathcal H_G([q_1, \ldots,q_b ]) = 0$, which implies $[q_1, \ldots,q_b ]=0$ and the claim is proven.
 
 The claim implies that $|W_F \cap \{F \ne 0\}| = 1 + |W_G\cap \{G \ne 0\}|$ and we are done by induction. 
 
\medskip

We now prove $(2)$.  By $(1)$, it is enough to show that $W_F':=W_F \cap \{F=0\}$ is finite.  Let $p \in W_F'$. We say that a hyperplane $L \subset \mathbb P^b$ is \emph{associated} to $p$ if:
\begin{itemize}
	\item $\det \mathcal H_F$ vanishes along $L$;
	\item $p \in L$; and
	\item if $T=F_{|L}$, then $\mathcal H_T(p)$ is trivial.
\end{itemize}

By Proposition \ref{prop:F(p)=0} we may assume that $p=[1,0,\ldots,0]$ and 
$$
F(x_0, \ldots, x_n)= x_0x_1^{n-1} + x_1^{n-1}S(x_2,\ldots,x_n) + x_1H(x_1,\ldots,x_n) + R(x_2,\ldots,x_n),
$$
where the degree of $H$ with respect to $x_1$ is at most $n-3$. Replacing $x_0$ with $x_0 - S$, we may assume 
\begin{equation}\label{eq:special}
F(x_0, \ldots, x_n)=x_0x_1^{n-1} + G(x_1, \ldots, x_n)= x_0x_1^{n-1} + x_1H(x_1,\ldots,x_n) + R(x_2,\ldots,x_n),
\end{equation}
where the degree of $H$ with respect to $x_1$ is at most $n-3$.

We claim that $L_p=\{x_1=0\}$ is an hyperplane associated to $p$. Indeed, to see that  $\det \mathcal H_F$ vanishes along $L_p$, note that for any point $q \in \{x_1=0\}$, $\mathcal H_F(q)$ has a trivial face (the one corresponding to $\partial_0$) and so we conclude by \cite[Corollary 1.5(d)]{GKZ}. Moreover, Lemma \ref{lem:H_F(p)=0} implies that if $T=F_{|{L_p}}$, then $\mathcal H_T(p)$ is trivial.

We now show that there exists a bijection between points of $W_F'$ and linear forms dividing $\det \mathcal H_F$, which implies that $W_F'$ is finite since $\det \mathcal H_F$ is not identically zero. 

Let $L'$ be a hyperplane associate to $p$. Assume by contradiction that $L' \ne L_p=\{x_1=0\}$. Then, after a  change of coordinates in $x_2,\ldots,x_n$, we can assume that $L'=\{\alpha x_1 - x_n=0\}$ for some $\alpha \in \mathbb C$. This gives
$$
T'(x_0,x_1,\ldots, x_{n-1}) :=F_{|L'}=x_0 x_1^{n-1} + G(x_1,\ldots,x_{n-1}, \alpha x_1)
$$
and so $\mathcal H_{T'}(p) \ne 0$, which is a contradiction.

Let $q \in W_F'$ be a point with associated hyperplane $L_q=L_p=\{x_1=0\}$. We want to show that $q=p$. Assume by contradiction that $q \ne p$. Since $q\in L_q$, after a  change of coordinates in $x_2, \ldots, x_n$ we may assume $q=[q_0,0,1,0,0\ldots,0]$ and we may write $F$ in the same form as in  \eqref{eq:special}. Since $\{x_1=0\}$ is a hyperplane associate to $q$, we have $\mathcal H_R([1,0,\ldots,0])=0$ and so by Lemma \ref{lem:H_F(p)=0} we get that $R$ does not depend on $x_2$. 
Let $m$ be the minimum $t\in \{1,\ldots,n\}$  such that there exists a monomial in $F$ of the form $x_1^tx_2x_{i_1}\cdots x_{i_{n-t-1}}$  where $i_j \ge 2$ for any $j=\{1, \ldots, n-t-1\}$. Since $F$ is honest, it depends on $x_2$. Thus, since $R$ does not depend on $x_2$, such $m\ge 1$ exists.
Since the degree of $H$ with respect to $x_1$ is at most $n-3$ we have $m \le n-2$, i.e.\ $n-m-1 \ge 1$. 
Let $i_2,\ldots,i_{n-m-1}\ge 2$ such that $x_1^mx_2x_{i_1}\cdots x_{i_{n-m-1}}$ is a monomial of $F$.
We  consider the matrix 
$$
A=\left(\partial_j \partial_\ell\partial_1^{m-1} \partial_{i_2} \cdots \partial_{i_{n-m-1}}F(q)\right)_{j, \ell = 0, \ldots,  b }.
$$ 
Then $A_{1i_1} \ne 0$. Since $\rk A \le 1$ and $i_1 \ge 2$, we get that $A_{i_1 i_1} \ne 0$ by symmetry,
i.e. the monomial $x_1^{m-1}x_2x_{i_1}^2x_{i_2}\cdots x_{n-m-1}$ appears in $F$ with non-zero coefficient, which contradicts the minimality of $m$. We hence conclude that $q=p$ and the claim follows.
\end{proof}

\section{Intersection forms} \label{s:intersection}

Let $M$ be a closed (i.e.\ compact and without boundary) oriented topological manifold of dimension $d=2n$. 
Consider the symmetric 
$n$-multilinear map defined by the cup product and evaluated on the orientation class:
$$
\phi_M : H^{n} \to \Z
$$
where  
$H=H^2_{tf}(M,\Z)$ denotes the singular cohomology group $H^2(M,\Z)$ modulo its torsion subgroup. Fix a basis $\underline h=(h_0,\ldots,h_b)$ of $H$.
Then we define a homogeneous $n$-form $F_M \in \Z[x_0,\ldots,x_b]$ by

$$
F_M=F_M(x_0,\ldots,x_b):= \sum_{\underline m} \binom{n}{\underline m} \phi_M(h_0^{m_0}\cdot\ldots\cdot h_b^{m_b})x_0^{m_0}\cdot\ldots\cdot x_b^{m_b},
$$
where the sum runs over all the $(b+1)$-tuples $\underline m=(m_0,\dots,m_b)$ such that $m_0,\ldots,m_b\ge 0$ and $\sum m_i=n$, and we denote
$$
\binom{n}{\underline m}= \frac{n!}{m_0! m_1! \cdots m_b!}.
$$

\begin{lemma}\label{lem:linear}Set-up as above. 
Consider the $(n-1)$-multilinear map $\phi_M(h_0,-,\ldots,-): H^{n-1} \to \Z$ with associated hypermatrix $A$. Then
$$
n!A=\mathcal H_{F}(h_0). 
$$
\end{lemma}
\begin{proof}
Let $i_2, \ldots, i_n\in \{0,\ldots,b\}$ such that if $m_k$ is the number of $i_j$'s equal to $k$ then $m_0+\ldots +m_b=n-1$.

Then $A_{(i_2,\ldots,i_n)}=\phi_M(h_0,h_{i_2}, \ldots, h_{i_n})$ and if $I=(0,i_2,\ldots,i_n)$, then 
$$\begin{aligned}
\mathcal H_{F}(h_0)_{(i_2,\ldots,i_n)}&=
\frac{\partial^{n-1}}{\partial_0^{m_0} \cdots \partial_b^{m_b}} F(1,0,\ldots,0)\\
&=\frac{\partial^{n-1}}{\partial_0^{m_0} \cdots \partial_b^{m_b}}\binom{n}{\underline{m}} \phi_M(h_0,h_{i_2,\ldots,h_{i_n}}) x_0 x_I(1,0\ldots,0)\\
&=n!\phi_M(h_0,h_{i_2}, \ldots, h_{i_n}).
\end{aligned}
$$ Thus, our claim follows.
\end{proof}

If $X$ is a compact complex manifold, then we denote by $F_X$ the $n$-form associated to the  topological manifold underlying $X$.

\begin{lemma}
If $X$ is a compact K\"ahler manifold, then $\mathcal H_{F_X}$ is honest. 
\end{lemma}
\begin{proof}
Assume there exists  $v \in H$ such that $\mathcal H_{F_X}(v)=0$. By Lemma \ref{lem:linear} this means that the multilinear map $\phi_X(v,-,\ldots,-)$ is trivial. This implies that $v=0$ because by Lefschetz theorem the map $\omega^{n-1}\colon H \to H^{\vee}$ is an isomorphism, where $\omega$ is a K\"ahler class.
\end{proof}

If $X$ is a compact K\"ahler theefold then the associated cubic form is non-degenerate by \cite[Proposition 16]{OV95}. But, this is not true in higher dimension, as the following example shows.

\begin{example}
Consider $X=\mathbb P^1 \times \mathbb P^3$ with the two projections $p_1$ and $p_2$. Then $H^2(X, \Z) = \Z \oplus \Z$ and we can take $H_1=p_1^*\mathcal O_{\mathbb P^1}(1)$ and $H_2=p_2^*\mathcal O_{\mathbb P^3}(1)$ as a basis. Let $x_0,x_1$ be the corresponding coordinate.  Then $F_X=4x_0x_1^3$ and $\mathcal H_{F_X}$ is a $2 \times 2 \times 2$ hypermatrix with the following faces  (up to a factor of 24):
$$
\begin{pmatrix}
0 & 0   \\
0 & x_1   \\
\end{pmatrix}, \quad  
\begin{pmatrix}
0 & x_1   \\
x_1 & x_0   \\
\end{pmatrix}.
$$
By \cite[Proposition XIV.1.7]{GKZ}, it follows that $F_X$ is degenerate. 
\end{example}

\section{Blow-ups}\label{s_blowups}

\begin{lemma}\label{lem:blowup}
Let $X$ be a smooth projective variety of dimension $n$ and let $f\colon Y \to X$ be the blow-up along a smooth subvariety $Z$ of $X$ of dimension $k < n$. 
Let $[E], \beta_1:=f^*\gamma_1, \ldots, \beta_b:=f^*\gamma_b$ be a basis of $H^2(Y,\Q)$ where $E$ is the exceptional divisor and $\gamma_1,\ldots,\gamma_b$ is a basis of $H^2(X,\Q)$.	
Then, with respect to this basis, we may write
$$
F_Y(x_0,\ldots,x_n)=ax_0^n +\sum_{i=1}^{n-k} x_0^{n-i}R_i + F_X(x_1, \ldots, x_b)
$$
where $a=E^n$ and $R_i \in \mathbb Q[x_1,\ldots,x_b]$ is a form  of degree $i$, for $i=1,\dots,n-k$. 

\end{lemma}
\begin{proof}
This follows easily from the projection formula. 
\end{proof}

\begin{proposition}\label{prop:blowuprank}
Let $X$ be a smooth K\"ahler manifold of dimension $n$ and let $f\colon Y \to X$ be the blow-up along a closed submanifold $Z$ of $X$ of dimension $k \le 2$. Let $p=[E] \in H^2(Y,\C)$ be the class of the exceptional divisor.
\begin{enumerate}
\item \label{k=0} If $n\ge 2$ and $k=0$ then $\rk \mathcal H_{F_Y}(p)=1$.
\item \label{k=1}  If $n\ge 3$ and  $k=1$ then $\rk \mathcal H_{F_Y}(p) \ge 2$. 
\item \label{k=2} If $n\ge 4$ and $k=2$ then, after a base change, we may write 
$$
F_Y(x_0,\ldots,x_b)=ax_0^n + x_0^{n-1}L(x_1,\ldots,x_b) + x_0^{n-2}Q(x_1,\ldots,x_b) + F_X(x_1, \ldots, x_b)
$$
where $a \in \mathbb Q$, $L$ is a linear form and $Q$ is a quardric form of rank $q$ for some positive integer $q$ such that $\rk \mathcal H_{F_Y}(p) \ge 2q$.

\end{enumerate}
\end{proposition}
\begin{proof}
We use the same notation as in Lemma \ref{lem:blowup}.

We first prove \eqref{k=0}. We have 
$$
F_Y(x_0,\ldots,x_n)=ax_0^n  + F_X(x_1, \ldots, x_b)
$$
where $a=E^n \ne 0$. Hence $\rk \mathcal H_{F_Y}(p)=1$.
Thus, \eqref{k=0} follows. 
\medskip 

We now prove \eqref{k=1}. We have 
$$
F_Y(x_0,\ldots,x_n)=ax_0^n+ x_0^{n-1}\cdot \left(\sum_{i=1}^{b}\lambda_ix_i \right )  + F_X(x_1, \ldots, x_b )
$$
where $a=E^n$ and $\lambda_i=Z \cdot \gamma_i$, for $i=1,\dots,b$. After taking a base change that fixes $(1,0,\ldots,0)$, we may assume $\lambda_i=0$ for any $i=2, \ldots ,b$, i.e.

$$
F_Y(x_0,\ldots,x_n)=ax_0^n+ \lambda_1 x_0^{n-1}x_1  + F_X(x_1, \ldots, x_b ),
$$
where $\lambda_1 \ne 0$ because we are blowing up a curve $Z$ in a K\"ahler manifold and so there exists at least one class in $H^2(X,\mathbb Z)$ with non-zero intersection with $Z$. Thus, it is easy to check that the subtensor given by 
$$(\partial_0^{n-3}\partial_i\partial_j)_{i,j=0,\dots,b}$$ 
has rank two,  which implies that $\rk \mathcal H_{F_Y}(p) \ge 2$. 

\medskip 

We finally prove \eqref{k=2}. Let $j\colon Z \hookrightarrow X$ be the inclusion. 
Let  $V \subset H^2(Z, \C)$ be the subspace generated by $j^*\gamma_1, \ldots, j^*\gamma_b$ and let $q$ be the rank of the quadratic form $Q$ obtained restrincting to $V$ the quadratic form of $H^2(Z,\mathbb C)$.
 Note that $q \ge 1$ because a K\"ahler class of $X$ restricts to a K\"ahler class on $Z$. 
After a base change that diagonalises $Q$, we can write
$$
F=F_Y(x_0,\ldots,x_b)=\frac{ax_0^n}{n!} + \frac{x_0^{n-1}}{(n-1)!}(c_1x_1+\ldots +c_bx_b) + \frac{x_0^{n-2}}{2(n-2)!}(x_1^2+\ldots+ x_q^2) + F_X(x_1, \ldots, x_b)
$$
for some    $c_1,\dots,c_b \in \C$.

Consider the $(b+1)\times (b+1) \times (b+1)$ subtensor $A$ of $\mathcal H_{F}(p)$ given by the slices $A_0, \ldots, A_b$ where 
 $$
A_h= ( \partial^{n-4}_0 \partial_h \partial_i \partial_j F(p))_{i,j=0,\dots,b}
 $$
 for $h=0,\dots,b$. 

It is enough to prove that $\rk A \ge 2q$, as it immediately implies that $\rk \mathcal H_{F}(p) \ge 2q$.
We have
$$
A_0 = 
\begin{pmatrix}
1 & c_1 & \ldots & c_q & c_{q+1} & \ldots & c_b \\
c_1 & 1  & 0 & \ldots & \ldots & \dots & 0 \\
\vdots & 0 & \ddots  & 0 & \cdots & \cdots & \vdots \\
c_q & 0 & \ldots & 1 & 0 & \ldots & 0  \\
c_{q+1} &\vdots & \vdots  &\vdots & 0 & \ldots & 0 \\
\vdots&\vdots & \vdots  &\vdots & \vdots & \ddots & \vdots \\
c_b & \cdots & \cdots & \cdots & 0 & \cdots & 0 \\
\end{pmatrix}$$
and
$$
A_h = 
\begin{pmatrix}
c_h & 0 & \ldots  & 1 & 0 & \ldots \\
0 & 0  & \ldots & 0 & 0 & \ldots \\
\vdots & \vdots & \ddots  & \vdots & \vdots & \ldots  \\
1 & 0 & \ldots  & 0 & 0 & \ldots  \\
0 & 0 & \ldots & 0 & 0 & \ldots \\
\vdots &\vdots & \vdots  &\vdots &\vdots & \ddots  \\
\end{pmatrix}
$$
for $1 \le h \le q$ where $A_{hh0}=A_{h0h}=1$.

It follows that  the matrices $A_1,\ldots, A_q$ are linearly independent and that for any $q$-uple $(\mu_1, \ldots, \mu_q) \in \C^q$ the rank of the matrix $A_0 + \sum_{j=1}^q \mu_j A_j$ is $q$ or $q+1$.
By Lemma \ref{lem:trick} we then have $\rk A \ge 2q$, as claimed. 
 \end{proof}

\begin{proof}[Proof of Theorem \ref{thm_main}] This is an immediate consequence of Theorem \ref{thm:W_F} and Proposition \ref{prop:blowuprank}.
\end{proof}

\begin{example}
Let $X$ be a complex projective manifold of dimension $n$ and $b_2(X)=1$. Let $f\colon Y \to X$ be the blow-up along a smooth curve $Z$ of $X$. Let $p=[E] \in H^2(Y,\C)$ be the class of the exceptional divisor. Then we can write 
$$
F_Y(x_0,\ldots,x_n)=\frac{a}{n!}x_0^n+ \frac{x_0^{n-1}x_1}{(n-1)!}  + x_1^n,
$$  
for some $a \in \mathbb Q$. If $n=3$, then 
$$
\mathcal H_{F_Y}(p)= 
\begin{pmatrix}
a & 1 \\
1 & 0 \\
\end{pmatrix}
$$
has rank 2. If $n=4$, then 
$$
\mathcal H_{F_Y}(p)= 
\left[ \begin{pmatrix}
a & 1 \\
1 & 0 \\
\end{pmatrix},
\begin{pmatrix}
1 & 0 \\
0 & 0 \\
\end{pmatrix}
\right]
$$
has rank at least 3 by Lemma \ref{lem:trick}. On the other hand, 3 is the maximal rank for a $2\times 2 \times 2$ tensor (see, for instance,  \cite[Section 3]{Kruskal}). Thus, $\mathcal H_{F_Y}(p)$ has rank $3$.
\end{example}

\section{Topological bounds and birational geometry}\label{s_volume}

The goal of this section is to present some open problems on the relationship between birational invariants and topological invariants of a smooth complex projective variety.  

\subsection{Volume}

We begin with the following question on the volume of a projective manifold:

\begin{question}\label{q:volume}
Let $M$ be a closed topological manifold of dimension $2n$. Is there a constant $C$, depending only on $M$, such that for any smooth complex projective variety $X$ whose underlying topological space is $M$, we have
\[ \vol(X) \le C? \]
\end{question}

For curves, the volume is a topological invariant, as we have $\vol(X) = 2g-2$, where $g$ is the genus. In dimension $2$, using the Bogomolov-Miyaoka-Yau inequality, we obtain a topological bound since $K_X^2 \leq 3c_2(X)$, implying that $\vol(X)$ is bounded by  topological invariants of $X$.

For dimension $3$, the following result holds:

\begin{theorem}[\cite{CT18}, Theorem 1.2]\label{thm:boundvolume}
Let $X$ be a smooth complex projective threefold. 

Then 
\[ \vol(X) \le 6b_2(X) + 36b_3(X). \]
\end{theorem}

An interesting consequence is that for smooth projective threefolds of general type sharing the same underlying 6-manifold, the volume takes only finitely many values. Moreover, it follows that the family of smooth projective threefolds of general type with bounded Betti numbers is birationally bounded.

For dimension $4$, if $K_X$ is nef and big, similar arguments using the Bogomolov-Miyaoka-Yau inequality yield a bound on $K_X^4$ in terms of the topology of $X$ (see \cite[Page 525]{Kotschick08}).

The proof of Theorem \ref{thm:boundvolume} uses in a fundamental way the classification of terminal singularities in dimension 3. It would be interesting to see an approach more suitable to be generalised to higher dimensional varieties.

\subsubsection{Sketch of the proof of Theorem \ref{thm:boundvolume}}
We can assume that $X$ is of general type, otherwise $\vol(X)=0$.  If $Y$ is a minimal model of $X$, then $\vol(X)=\vol(Y)=K_Y^3$.
The Bogomolov-Miyaoka-Yau inequality and a singular version of the Riemann-Roch theorem for threefolds imply
$$
K_Y^3 \le 3K_Y \cdot c_2(Y) \le 3(12b_3(X) + \Xi(Y)),
$$ 
where $\Xi(Y)$ is an integer depending on the singularities of $Y$.
The conclusion follows now from the fact that  $\Xi(Y) \le 2b_2(X)$ (see \cite[Proposition 3.3]{CZ14}). Note that this last step uses the classification of terminal threefolds singularities. 

\subsection{Number of minimal models}
We now examine how the topology of a complex projective manifold influences the number of its minimal models:

\begin{question}[\cite{CL14}, Question 1.3]\label{q:minimal_models}
Let $X$ be a smooth complex projective variety of general type. Does there exist a positive constant $C$, depending only on the topology of $X$, bounding the number of minimal models of $X$?
\end{question}

In dimension $2$, minimal models are unique. For threefolds, this was resolved in \cite[Corollary 3]{MST}, which shows that the function counting minimal models is constructible on any family of varieties of general type. Birational boundedness then implies the existence of a bound $N(c)$ depending only on the volume $c$, and since by Theorem \ref{thm:boundvolume}, the volume is topologically bounded for threefolds, so is the number of minimal models.
Similarly, a positive answer to Question \ref{q:volume} in dimension $n$ would imply a positive answer to Question \ref{q:minimal_models} in dimension $n$.

Note that \cite[Corollary 3]{MST} provides only non-effective bounds. Explicit bounds would be highly desirable: see \cite{Mar22} for some progress regarding threefolds with small Picard number and \cite{kim25} for some related result on threefolds of Fano type.  

\subsection{Chern numbers}

Chern and Hodge numbers are fundamental numerical invariants of Kähler manifolds. Addressing a question by Hirzebruch, all linear combinations of Chern and Hodge numbers that are topological invariants for smooth projective varieties were identified in \cite{Kotschick08,Kotschick12,KS13}.

Generalising Hirzebruch's question, Kotschick asked if the topology of the underlying smooth manifold determines Chern numbers up to finite ambiguity (this is clear for Hodge numbers due to the Hodge decomposition theorem).

For compact complex  manifolds of dimension $n$, $c_n$ coincides with the Euler characteristic, a topological invariant. Furthermore, \cite{LW90} proved that $c_1c_{n-1}$ is determined by Hodge numbers and thus bounded by Betti numbers.

In general, the following is known:

\begin{theorem}
[\cite{ST16}, Theorem 1]\label{thm:ST}
In dimension $4$, the Chern numbers $c_4$, $c_1c_3$, and $c_2^2$ of a complex projective variety are determined up to finite ambiguity by the underlying smooth manifold. For dimension $n \geq 5$, only $c_n$ and $c_1c_{n-1}$ are similarly determined.
\end{theorem}

Note that the examples constructed in \cite{ST16} to establish Theorem \ref{thm:ST} are projective bundles.

In dimension $1$, $c_1 = 2 - 2g$ is a topological invariant. For surfaces, if $X$ and $Y$ are homeomorphic, either $c_1^2(X) = c_1^2(Y)$ or $c_1^2(X) = 4c_2(Y)-c_1^2(Y)$ depending on their orientation (see \cite{Kotschick08}). However, if they are diffeomorphic, then $c_1^2(X) = c_1^2(Y)$.

\medskip
Thus, the main remaining open case is:

\begin{question}[Kotschick]
Does $c_1^3 = -K_X^3$ assume only finitely many values for smooth projective structures on a given 6-manifold?
\end{question}

Recent progress on this problem is described in \cite{CT18, ST19, ST20, CC24}. 

\subsection{Topology and MMP}

Given a smooth projective threefold  $X$ and a minimal model program $f\colon X \dashrightarrow Y$, it is natural to ask which topological invariants of $Y$ are determined by those of $X$. It is known that the Betti numbers of $Y$  are determined up to finite ambiguity by those of $X$, with the delicate case of \( b_3 \) treated in~\cite{Chen19}.

We propose a more ambitious question:

\begin{question}\label{q:cubic}
Let $X$ be a Kähler threefold with cubic form $F_X$ and first Pontryagin class $p_1(X)$. Define $\mathcal{P}$ as the set of triples $(H^2(Y,\mathbb{Z}), F_Y, p_1(Y))$, up to isomorphism, arising from minimal models $X \dashrightarrow Y$. 

Is the set $\mathcal{P}$ determined up to finite ambiguity solely by $(H^2(X,\mathbb{Z}), F_X, p_1(X))$?
\end{question}
Partial results and consequences of Question \ref{q:cubic} have been obtained in \cite{CT18, ST19, ST20, CC24}. In particular, the question admits a positive answer if the MMP $X \dashrightarrow Y$ is composed by divisorial contractions to points and to smooth curves, assuming that the discriminant of $F_X$ is non-zero (i.e. the hypersurface defined by $F_X=0$ is non-singular). It would be very interesting to gain a better understanding for flips and in the discriminant zero case.

\medskip

In higher dimension, we propose the following question with the goal of improving Proposition \ref{prop:blowuprank}: 

\begin{question}
Let $X$ be a smooth K\"ahler manifold of dimension $n$ and let $f\colon Y \to X$ be the blow-up along a closed submanifold $Z$ of $X$ of dimension $k$. Let $p=[E] \in H^2(Y,\C)$ be the class of the exceptional divisor.

What is the rank of $\mathcal H_{F_X}(p)$?
\end{question}

\bibliographystyle{amsalpha}
\bibliography{Library}
\end{document}